\documentclass[11pt]{article}
\usepackage{geometry}                		
\usepackage[parfill]{parskip}    		

\usepackage{amssymb}
\usepackage{amsmath}

\usepackage{palatino}

\usepackage{graphicx}
\usepackage[english]{babel}
\usepackage{amscd}
\usepackage{amsthm}

\usepackage{color}

\usepackage{hyperref}

\newtheorem{theorem}{Theorem}[section]
\newtheorem{lemma}[theorem]{Lemma}

\newtheorem{conjecture}[theorem]{Conjecture}

\newtheorem{problem}[theorem]{Problem}

\DeclareMathOperator*{\HC}{\mathcal{H}_{\mathcal{C}}}
\DeclareMathOperator*{\Hy}{\mathcal{H}}
\DeclareMathOperator*{\Cy}{\mathcal{C}}
\DeclareMathOperator*{\Bin}{\textit{Bin}}

\title{On the Shannon entropy of the number of vertices with zero in-degree in randomly oriented hypergraphs}

\author{
Christos Pelekis\thanks{Institute of Mathematics, Czech Academy of Sciences, \v{Z}itna 25, Praha 1, Czech Republic. 
Research supported by GA\v{C}R project 18-01472Y and RVO: 67985840. E-mail: pelekis.chr@gmail.com} 		
}

\begin{document}
	\maketitle
	
\begin{abstract}
Suppose that you have $n$ colours and $m$  mutually independent dice, each of which has $r$ sides. 
Each dice lands on any of its sides with equal probability. You may colour the sides of each die in any way you wish, but there is one restriction: you are not allowed to use the same colour more than once on the sides of a die. Any other colouring is allowed. Let $X$ be the number of different colours that you see after rolling the dice. How should you colour the sides of the dice in order to maximize the Shannon entropy of $X$? In this article we investigate this question. We show that the entropy of $X$ is at most $\frac{1}{2} \log(n) + O(1)$ and that the bound is tight, up to a constant additive factor, in the case of there being equally many coins and colours. Our proof employs  the differential entropy bound on discrete entropy, along with a lower bound on the entropy of binomial random variables whose outcome is conditioned to be an even integer. We conjecture that the entropy is maximized when the colours are distributed  over the sides of the dice as evenly as possible. 
\end{abstract}

\noindent {\emph{Keywords}: Shannon entropy; randomly oriented hypergraphs

\section{Prologue and main results}\label{s1}

This work is motivated by the following entropy-maximization problem: 
Fix positive integers $m,n,r$ such that $m\ge n>r\ge 2$. 
Suppose you are given $n$ colours and $m$ mutually independent dice, each of which has $r$ sides. Each dice lands on any of its sides with equal probability.  
You can colour the sides of the dice in any way you want, but there is only one restriction: you are not allowed to use the same colour more than once on the sides of a die. All other colourings are allowed. Let $X$ be the number of different colours  that you see after rolling the dice. In what way should you colour the sides of the dice in order to maximize the Shannon entropy of $X$? 

The \emph{Shannon entropy} (or \emph{entropy}, for short) of a random variable $X$ that takes values on a finite set $S$ is defined as 
\[ H(X) =- \sum_{s\in S} \mathbb{P}[X=s]\cdot \log \mathbb{P}[X=s] , \]
with the convention $0\log 0 = 0$. 
Throughout the text, $\log(\cdot)$ denotes logarithm with base $2$.  Shannon entropy may be thought of  as the "amount of information", or the "amount of surprise", that is evidenced by a random variable and, in a certain sense, random variables with large entropy are less "predictable". 
Entropy enjoys several interesting properties which render itself as a useful tool for several problems in  enumeration, statistics and theoretical computer science, among several others.   
We refer the reader to \cite{CoverThomas, Johnson} for excellent textbooks on the topic. 
A central theme that motivates the development of the theory of entropy concerns the so-called \emph{maximum entropy principle}: within a given class of random variables, find one that has maximum entropy (see  \cite{CoverThomas} for a whole chapter devoted to the topic).  

It is well-known that for any random variable taking values in a finite set $S$, it holds 
\[ H(X) \le \log(|S|). \] 
This is a consequence of Jensen's inequality. Notice that the bound is attained by the random variable that takes each value from $S$ with equal probability. 
The following result, due to James  Massey, is referred to as the \emph{the differential entropy bound on discrete entropy} and may be seen as a refinement of the aforementioned bound.  \\

\begin{theorem}[Massey]\label{massey_bound}
Let $X$ be a discrete random variable with finite variance, denoted $\text{Var}(X)$. Then 
\[ H(X) \le \frac{1}{2}\log\left(2\pi e (\text{Var}(X) + 1/12)\right). \]
\end{theorem}

In other words, bounds on the variance of discrete random variables imply bounds on their entropy.  
A proof of Theorem~\ref{massey_bound} can be found in \cite{Massey} (see also \cite[p. 258]{CoverThomas})  and a refinement can be found in \cite{Mow}.

This work is concerned with the problem of maximizing the entropy of the random variable that counts the  number of different colours after a roll of $m$ fair dice whose $r$ sides have been coloured using $n$ colours, subject to the condition that it is not allowed to use a colour more than once on the sides of a die; we refer to this condition as a \emph{proper colouring}.  
The random variable that counts the number of different colours after a toss of $m$ properly coloured coins (i.e., when $r=2$) has been previously studied in \cite{Pelekis, Pel_Schauer}, in the setting of maximising its median. Similarly to the setting of the median, 
we conjecture that a proper colouring over the dice that maximizes Shannon entropy is such that the colours have been distributed as evenly as possible over the sides of the dice. 
In order to be more precise, we need some extra piece of notation. 

Let the positive integers $m,n,r$ be such that $m\ge n >r\ge 2$. 
Suppose that $\mathcal{C}$ is a configuration consisting of $m$ dice all of whose $r$ sides have been properly coloured using $n$ colours.  Let $X_{\mathcal{C}}$ be the number of different colours after a roll of the dice. 
One may associate a hypergraph, $\HC=(V,\mathcal{E})$, to this configuration: for every colour put a vertex in $V$ and for every (properly) coloured die put an edge in $\mathcal{E}$ containing all vertices corresponding to the colours on the sides of the die.  
Notice that $|V|=n$ and $|\mathcal{E}|=m$. Moreover, the hypergraph $\HC$ may have isolated vertices and, since the same coloured dice may appear more than once in the configuration $\mathcal{C}$, it may have edges that appear more than one time in $\mathcal{E}$; i.e., it is a \emph{multi-hypergraph}.  
Notice also that every edge $E\in\mathcal{E}$ has cardinality $r$ or, in other words, the hypergraph $\HC$ is $r$-\emph{uniform}. 
A $2$-uniform (multi)hypergraph is just a (multi)graph. Here and later, the class consisting of all  $r$-uniform  multi-hypergraphs on $n$ vertices and $m$ edges is denoted by  $\mathcal{D}_{n,m,r}$.  
The class $\mathcal{D}_{n,n,r}$, i.e, the class consisting of all $r$-uniform hypergraphs having $n$ vertices and $n$ edges, will be of particular interest. 
When $r=2$, we write $\mathcal{G}_n$ instead of $\mathcal{D}_{n,n,2}$. 

Now rolling the dice corresponds to choosing an element from each edge $\HC$ uniformly at random, where each choice is done independently of all previous choices;   we refer to this sampling procedure  by saying that each edge in $\Hy_{\mathcal{C}}$ is \emph{randomly oriented towards one of its elements}.   
Here and later, the phrase \emph{random orientation on the edges} of a hypergraph $\Hy$  expresses the fact that each edge in $\Hy$ is oriented towards one of its elements with probability $\frac{1}{r}$, and independently of all other edges.  
For every $v\in V$, let $\deg_{\Hy}^{-}(v)$ denote the \emph{in-degree} of $v$, i.e.,  the number of edges in $\mathcal{E}$ that are oriented towards $v\in V$ after a random orientation on the edges of $\Hy$. Notice that 
\[ X_{\mathcal{C}} = |\{v\in V : \deg_{\Hy}^{-}(v) >0\}| . \]
In other words, $X_{\mathcal{C}}$ is the number of vertices with non-zero in-degree after a random orientation on the edges of $\HC$ and the above mentioned question on coloured dice can be equivalently expressed as follows. \\

\begin{problem}\label{prbl:1}
Fix positive integers $n,m,r$ such that $m\ge n>r\ge 2$. For every $\Hy\in\mathcal{D}_{n,m,r}$ let $X_{\Hy}$ be the random variable that counts the number of vertices with non-zero in-degree after a random orientation on the edges of $\Hy$. Find a hypergraph $\Hy\in\mathcal{D}_{n,m,r}$ such that 
\[ H(X_{\Hy}) \ge H(X_{\mathcal{F}}), \; \text{ for all } \; \mathcal{F}\in \mathcal{D}_{n,m,r} . \]
\end{problem}
We conjecture that the hypergraph that maximizes entropy is such that the degrees of its vertices are as equal as possible. More precisely, we believe that the following holds true. Given a hypegraph $\Hy=(V,\mathcal{E})$ and a vertex $v\in V$, we denote by $\deg_{\Hy}(v)$ the number of edges in $\mathcal{E}$ that contain $v$. \\

\begin{conjecture}\label{main_conjecture}
Let the positive integers $m,n,r$ be such that $m\ge n >r\ge 2$. 
A hypergraph $\Hy=(V,\mathcal{E})$ from $\mathcal{D}_{n,m,r}$ for which it holds
\[ H(X_{\Hy}) \ge H(X_{\mathcal{G}}), \; \text{ for all } \; \mathcal{G}\in \mathcal{D}_{n,m,r} \]
is such that 
\begin{equation}\label{deg_conj}
|\deg_{\Hy}(v_1)-\deg_{\Hy}(v_2)|\le 1, \; \text{ for all }\; v_1,v_2\in V.
\end{equation}
\end{conjecture}

In other words, the colours should be distributed over the dice as evenly as possible. 
In this note we provide an upper bound on the entropy of $X_{\Hy}$, for $\Hy\in \mathcal{D}_{n,m,r}$.  
More precisely, we obtain the following. \\

\begin{theorem}\label{main_theorem}
Fix  $\Hy\in \mathcal{D}_{n,m,r}$ and let $X_{\Hy}$ be the number of vertices with non-zero in-degree after a random orientation on the edges of $\Hy$. Then
\[ H(X_{\Hy}) \le \frac{1}{2}\log(n) + \frac{1}{2}\log(\pi e) . \]
\end{theorem}

We prove Theorem~\ref{main_theorem} in Section~\ref{sec:2}.
Moreover, we show that the bound is tight, up to an additive constant factor, when $m=n$ and $r=2$. 
That is, we show that there exist graphs in $G\in\mathcal{G}_n$ for which $H(X_G)$ is of order $\frac{1}{2}\log(n)$. In fact, one can find several graphs $G\in\mathcal{G}_n$ for which the entropy of $X_G$ is of order $\frac{1}{2}\log(n)$. For example, if $n$ is even, one can take $G$ to be the union of $n/2$ vertex-disjoint double edges. The degree sequence of $G$ satisfies \eqref{deg_conj}, but $G$ is not connected. An example of a connected graph $G\in\mathcal{G}_n$ for which $X_G$ has entropy of order $\frac{1}{2}\log(n)$ is a star on $n$ vertices, plus an edge joining any two of its vertices; however, this graph does not satisfy \eqref{deg_conj}. It would therefore be interesting to know whether there exist \emph{connected} graphs $G$, whose degree sequence satisfies \eqref{deg_conj}, and are such that  $H(X_{G})$ is of order $\frac{1}{2}\log(n)$. The next result shows that cycles on $n$ vertices are examples of such graphs. Moreover, it is not difficult to see that the entropy corresponding to a cycle is slightly larger than the entropies corresponding to the aforementioned examples.  \\

\begin{theorem}\label{fix_r}
Let $C_n$ denote a cycle on $n\ge 3$ vertices. Let $X_{C_n}$ be the number of vertices with non-zero in-degree after a random orientation on the edges of $C_n$. Then
\[ H(X_{C_n}) \ge  \frac{1}{2}\log(n) + \frac{1}{2}\log(\pi e) -\frac{3}{2}-\frac{1}{2\ln(2)(n-1)}. \] 
\end{theorem}

We conjecture that  $H(X_{C_n})\ge H(X_G)$ for all $G\in\mathcal{G}_n$. 
We prove Theorem~\ref{fix_r} in Section~\ref{sec:3}.  The proof employs a lower bound on the entropy of a binomial random variable whose outcome is conditioned to be an even positive integer. 
Our article ends with Section~\ref{sec:4} in which we state a conjecture. 

\section{Proof of Theorem~\ref{main_theorem}}\label{sec:2}

We show that $\text{Var}(X_{\Hy}) \le n/4$. The result then follows from Theorem~\ref{massey_bound}. 
For every vertex $v\in \Hy$, let $\mathbf{I}_v$ be the indicator of the event $\{\deg_{\Hy}^{-}(v) >0\}$ and notice that $X_{\Hy} = \sum_{v\in V} \mathbf{I}_v$. We may therefore write 
\begin{eqnarray*}  
\text{Var}(X_{\Hy})  &=& \sum_{v_1,v_2\in V} \left( \mathbb{E}[\mathbf{I}_{v_1} \cdot \mathbf{I}_{v_2}] -   \mathbb{E}[\mathbf{I}_{v_1}] \cdot  \mathbb{E}[\mathbf{I}_{v_2}]\right)  \\ 
&=& \sum_{v\in V} \left( \mathbb{E}[\mathbf{I}_{v}] - \mathbb{E}[\mathbf{I}_{v}]^2\right) + \sum_{v_1\neq v_2}  \left( \mathbb{E}[\mathbf{I}_{v_1} \cdot \mathbf{I}_{v_2}] -   \mathbb{E}[\mathbf{I}_{v_1}] \cdot  \mathbb{E}[\mathbf{I}_{v_2}]\right) .
\end{eqnarray*} 
We now show that, whenever $v_1\neq v_2$, we have 
\begin{equation}\label{neg_corr}
\mathbb{E}[\mathbf{I}_{v_1} \cdot \mathbf{I}_{v_2}] -   \mathbb{E}[\mathbf{I}_{v_1}] \cdot  \mathbb{E}[\mathbf{I}_{v_2}] \le 0 , 
\end{equation}
or, in other words, the indicators $\mathbf{I}_{v_1}$ and $\mathbf{I}_{v_2}$ are negatively correlated. 
This is clearly true when there exists no edge $E\in \Hy$ such that $\{v_1,v_2\}\subset E$, and we may therefore assume that $v_1$ and $v_2$ are both elements of some edge in $\Hy$. 
Let $\mathcal{E}_{v_1}$ be the class consisting of those edges in $\Hy$ that contain $v_1$ and do not contain $v_2$ and, similarly, let $\mathcal{E}_{v_2}$ be the class consisting of those edges in $\Hy$ that contain $v_2$ and do not contain $v_1$. Finally, let $\mathcal{E}_{v_1,v_2}$ be the subset of the edges in $\Hy$ that contain both $v_1$ and $v_2$. 
Notice that  
\[ \mathbb{E}[\mathbf{I}_{v_1}]\cdot   \mathbb{E}[\mathbf{I}_{v_2}]  = \left( 1- \left(\frac{r-1}{r}\right)^{\deg_{\Hy}(v_1)} \right)\cdot \left( 1- \left(\frac{r-1}{r}\right)^{\deg_{\Hy}(v_2)} \right) \]
and we proceed by working out the term $\mathbb{E}[\mathbf{I}_{v_1} \cdot \mathbf{I}_{v_2}]$. 
 
Let $A_1$ be the event "\emph{there is no edge in $\mathcal{E}_{v_1,v_2}$ which is oriented towards either $v_1$ or $v_2$}", let $A_2$ be the event "\emph{no edge from $\mathcal{E}_{v_1,v_2}$ is oriented towards $v_2$, but some edge from $\mathcal{E}_{v_1,v_2}$ is oriented towards $v_1$}", let $A_3$ be the event "\emph{no edge from $\mathcal{E}_{v_1,v_2}$ is oriented towards $v_1$, but some edge from $\mathcal{E}_{v_1,v_2}$ is oriented towards $v_2$}", and, finally, let $A_4$ be the event "\emph{some edge from $\mathcal{E}_{v_1,v_2}$ is oriented towards $v_1$ and some other edge from $\mathcal{E}_{v_1,v_2}$ is oriented towards $v_2$}". In particular, notice that the event $A_4$ has non-zero probability if and only if $|\mathcal{E}_{v_1,v_2}|\ge 2$. If $|\mathcal{E}_{v_1,v_2}|=1$, then $A_4$ is empty and so $\mathbb{P}(A_4)=0$. Moreover, when $r=2$ the event $A_1$ is empty as well; thus $\mathbb{P}(A_1)=0$. 

Now we may write   

\[  \mathbb{E}[\mathbf{I}_{v_1} \cdot \mathbf{I}_{v_2}] =\sum_{i=1}^{4} \mathbb{P}(A_i)\cdot \mathbb{P}(\mathbf{I}_{v_1} \cdot \mathbf{I}_{v_2} =1 | A_i) . \]
Let $d_1 = |\mathcal{E}_{v_1}|, d_2 = |\mathcal{E}_{v_2}|$ and $d_3=|\mathcal{E}_{v_1,v_2}|$ 
and notice that $\deg_{\Hy}(v_1)=d_1+d_3$ and $\deg_{\Hy}(v_2)=d_2+d_3$. We compute  
\[ \mathbb{P}(A_1)\cdot \mathbb{P}(\mathbf{I}_{v_1} \cdot \mathbf{I}_{v_2} =1 | A_1) = \left(\frac{r-2}{r}\right)^{d_3} \cdot \left(1 - \left(\frac{r-1}{r}\right)^{d_1}   \right)\cdot \left(1 - \left(\frac{r-1}{r}\right)^{d_2}   \right).  \]
Moreover, since  
\[\mathbb{P}(A_2) = \mathbb{P}(A_3) = \sum_{i=1}^{d_3}\binom{d_3}{i}\left(\frac{1}{r}\right)^i \left(\frac{r-2}{r}\right)^{d_3-i},  \]
the binomial theorem yields 
\[ \mathbb{P}(A_2)\cdot \mathbb{P}(\mathbf{I}_{v_1} \cdot \mathbf{I}_{v_2} =1 | A_2) =  \frac{1}{r^{d_3}} \left((r-1)^{d_3} -(r-2)^{d_3}\right)\cdot \left( 1- \left(\frac{r-1}{r}\right)^{d_2}\right)   \]
as well as 
\[ \mathbb{P}(A_3)\cdot \mathbb{P}(\mathbf{I}_{v_1} \cdot \mathbf{I}_{v_2} =1 | A_3) =  \frac{1}{r^{d_3}} \left((r-1)^{d_3} -(r-2)^{d_3}\right) \cdot \left( 1- \left(\frac{r-1}{r}\right)^{d_1}\right).   \]
Finally, notice that 
\[ \mathbb{P}(A_4)\cdot \mathbb{P}(\mathbf{I}_{v_1} \cdot \mathbf{I}_{v_2} =1 | A_4) = \mathbb{P}(A_4) = 1-\mathbb{P}(A_1)-\mathbb{P}(A_2)-\mathbb{P}(A_3). \]
Now straightforward calculations show that   
\[ 
\mathbb{E}[\mathbf{I}_{v_1} \cdot \mathbf{I}_{v_2}] -   \mathbb{E}[\mathbf{I}_{v_1}] \cdot  \mathbb{E}[\mathbf{I}_{v_2}] = \left(\frac{r-2}{r}\right)^{d_3}\left(\frac{r-1}{r}\right)^{d_1+d_2}  - \left(\frac{r-1}{r}\right)^{d_1+d_2+ 2d_3}  \le 0.
\]
Thus \eqref{neg_corr} holds true and we have shown  
\[ \text{Var}(X_{\Hy})  \le \sum_{v\in V} \left( \mathbb{E}[\mathbf{I}_{v}] - \mathbb{E}[\mathbf{I}_{v}]^2\right). \] 
Now write 
\begin{eqnarray*} 
\mathbb{E}[\mathbf{I}_{v}] - \mathbb{E}[\mathbf{I}_{v}]^2 &=&  \left(1-\left(\frac{r-1}{r}\right)^{\deg_{\Hy}(v)}\right) - \left(1-\left(\frac{r-1}{r}\right)^{\deg_{\Hy}(v)}\right)^2  \\
&=& \left(\frac{r-1}{r}\right)^{\deg_{\Hy}(v)}\left(1 - \left(\frac{r-1}{r}\right)^{\deg_{\Hy}(v)}\right) \\
&\le& \frac{1}{4},
\end{eqnarray*}
where the last estimate follows from the inequality $x(1-x)\le 1/4$, for $x\in[0,1]$. 
Hence $\text{Var}(X_{\Hy})  \le n/4$ and Theorem~\ref{main_theorem} follows from Theorem~\ref{massey_bound}.

\section{Proof of Theorem~\ref{fix_r}}\label{sec:3}

Throughout this section, we denote by $\Bin(n,1/2)$ the binomial distribution of parameters $n$ and $1/2$ and we occasionally identify a random variable with its distribution. Given two random variables $Z,W$, the notation $Z\sim W$ indicates that they have the same distribution. 
Recall that  we work within the class $\mathcal{G}_{n}$ and that, given $G\in \mathcal{G}_{n}$, we denote by $X_G$ the number of vertices with non-zero in-degree after a random orientation on the edges of $G$. 
A non-negative integer which is equal to zero $\textit{mod}\; 2$ is referred to as an \emph{even} integer.  
Finaly,  $\Bin(n,e)$ denotes a $\Bin(n,1/2)$ random variable conditional on the event that it is \emph{even}.  
In particular, notice that 
\[ \mathbb{P}[\Bin(n,e)=k] = \binom{n}{k}\frac{1}{2^{n-1}}, \text{ for even } k, \]
a fact that is immediate upon observing that the probability that a $\Bin(n,1/2)$ random variable is even equals $1/2$ (see also \cite[Lemma 1]{Pelekis} for a more general result). 
We begin by showing that the entropy of $X_{C_n}$, where $C_n$ denotes a cycle on $n$ vertices, equals the entropy of a $\Bin(n,e)$ random variable. 
The following result is a direct consequence of \cite[Theorem 4]{Pel_Schauer}, but we include here an independent proof (which borrows ideas from \cite{Pel_Schauer}) for the sake of completeness.\\ 

\begin{lemma}\label{coin_lemma}
Let $C_n$ be a cycle on $n\ge 3$ vertices. Then $H(X_{C_n}) = H(\Bin(n,e))$.
\end{lemma}
\begin{proof}  
Let $Z$ be the number
of vertices with zero in-degree after the random orientation on the edges of $C_n$, and let $E$
be the number vertices of even in-degree. Notice that $X_{C_n} = n-Z$ and thus 
\[  H(X_{C_n}) = H(Z). \]
For $i\in\{1,2\}$, let $V_i:= \{v\in C_n : \deg_{C_n}^{-}(v)=i\}$.  The in-degree sum formula yields 
\[ n =\sum_{v\in C_n} \deg_{C_n}^{-}(v) = |V_1| + 2|V_2| . \]
Since $C_n$ has $n$ vertices, we also have 
\[ n = Z + |V_1| + |V_2| . \] 
If we subtract the last two equations we get $Z = |V_2|$ and thus, since $E=Z+|V_2|$, we conclude  
\[ E = 2\cdot Z . \]
In particular, this implies that $E$ is even and 
\[ H(Z) = H(E) . \]
Hence it is enough to determine the entropy of $E$. 
We claim that 
\[\mathbb{P}(E=2k) = \binom{n}{2k}\frac{1}{2^{n-1}}, \text{ for } k\le n/2.\]
This is clearly true for $k=0$, so assume that $k>0$ is such that $k\le n/2$. Now notice that, if we fix an orientation of the edges of $C_n$ for which $E=2k$, then between any two vertices with zero in-degree there exists a vertex whose in-degree equals two and, conversely, between any two vertices with in-degree equal to two there exists a vertex whose in-degree equals zero. 
This implies that, if we fix $2k$ vertices $\{v_1,...,v_{2k}\}$, whose in-degree is even, 
and we fix the in-degree of vertex $v_1$, say $\deg_{G}^{-}(v_1)=0$,  then the in-degrees of all other vertices are determined (that is, $v_2\in V_2, v_3\in V_0, \ldots, v_{2k}\in V_2$ and all other vertices have in-degree one). Since there are two ways to choose the in-degree of $v_1$, the claim follows. 
Summarising, we have shown that $E$ has the same distribution as a $\Bin(n,e)$ random variable and so their entropies are equal. 
\end{proof}

Using Lemma~\ref{coin_lemma}, we obtain a lower bound on the entropy of $X_{C_n}$ which is expressed in terms of the entropy of a binomial random variable. \\

\begin{lemma}\label{cycle_entropy}
Let $C_n$ be a cycle of $n\ge 3$ vertices. Then 
\[ H(X_{C_n}) \ge H(\Bin(n-1,1/2))  -1 . \]
\end{lemma}
\begin{proof}
Lemma~\ref{coin_lemma} implies that it is enough to show 
\[ H(\Bin(n,e)) \ge H(\Bin(n-1,1/2))-1 .\]
We first show that an outcome of $\Bin(n,e)$ can be obtained as follows: First draw from a $\Bin(n-1,1/2)$ random variable. If the outcome is even, then add $0$ the outcome. If the outcome is odd, then add $1$. 
Formally, let $X\sim\Bin(n-1,1/2)$ and define the random variable $\delta_X$ by
\[
\delta_{X} =
\begin{cases}
0, &\text{ if } \; X \; \text{ is even} \\
1, &\text{ if } \; X \; \text{ is odd}.
\end{cases}
\]
We claim that $\Bin(n,e)\sim X + \delta_X$. To prove the claim, first notice that $X + \delta_X$ is always even and fix an even integer $k$ from $\{0,1,\ldots,n\}$. 
Then, using the convention $\binom{a}{b}=0$, whenever $b<0$ or $a<b$, and the relation $\binom{n}{k} =\binom{n-1}{k} + \binom{n-1}{k-1}$, we may write 
\begin{eqnarray*} 
\mathbb{P}(X + \delta_X =k) &=& \mathbb{P}(\Bin(n-1,1/2)=k) +  \mathbb{P}(\Bin(n-1,1/2)=k-1) \\
&=& \binom{n}{k} \frac{1}{2^{n-1}}\\
&=& \mathbb{P}(\Bin(n,e)=k),
\end{eqnarray*}
and the claim follows. Hence it is enough to show 
\[ H(X+ \delta_X) \ge H(\Bin(n-1,1/2))  -1. \]
Now, using the chain rule for entropy, we have 
\[ H(X+\delta_X , X) = H(X)+ H(X+\delta_X|X) = H(X+\delta_X) + H(X|X+\delta_X) \] 
and therefore  
\begin{eqnarray*} 
H(X+\delta_X)  &=& H(X)+ H(X+\delta_X|X) - H(X|X+\delta_X) \\
&=& H(\Bin(n-1,1/2)) + H(X+\delta_X|X) - H(X|X+\delta_X),
\end{eqnarray*}
which in turn implies that it is enough to show 
\[ H(X+\delta_X|X) - H(X|X+\delta_X) \ge -1 . \]
Since $X$ determines $X+\delta_X$, we have $H(X+\delta_X|X)=0$ and we are left with showing 
\[ H(X|X+\delta_X) \le 1 . \]
To this end, write  
\[
H(X|X+\delta_X) = \sum_{k \text{ even}}H(X | X+\delta_X = k) \cdot \mathbb{P}(X+\delta_X = k).
\]
Now, conditional on the event $\{X+\delta_X = k\}$, it follows $X\in \{k-1,k\}$ and therefore  
\[ H(X | X+\delta_X = k) \le \log 2 = 1 . \]
This yields   
\[ H(X|X+\delta_X) \le \sum_{k \text{ even}}  \mathbb{P}(X+\delta_X = k) =1, \]
as desired. 
\end{proof}

The proof of Theorem~\ref{fix_r} is almost complete. 

\begin{proof}[Proof of Threorem~\ref{fix_r}]
It is known (see \cite[p. 4]{Adel_et_al})  that 
\[
H(\Bin(n-1,1/2)) \ge \frac{1}{2}\log(n-1)  + \frac{1}{2}\log(\pi e) -\frac{1}{2}.
\]
Hence, using the inequality $\log(n-1)\ge \log(n) - \frac{1}{(n-1)\ln(2)}$, we conclude
\[  H(\Bin(n-1,1/2))\ge \frac{1}{2}\log(n) + \frac{1}{2}\log(\pi e) -\frac{1}{2}-\frac{1}{2\ln(2)(n-1)} \]
and Lemma~\ref{cycle_entropy} yields 
\[H(X_{C_n}) \ge   \frac{1}{2}\log(n) + \frac{1}{2}\log(\pi e) -\frac{3}{2}-\frac{1}{2\ln(2)(n-1)}.    \] 
The result follows. 
\end{proof}

\section{A conjecture}\label{sec:4}

In this section we define a hypergraph from the class $\mathcal{D}_{n,n,r}$  which, we believe, maximizes the entropy of the number of vertices of non-zero in-degree, after a random orientation on its edges.   
 
Fix a positive integer $r\ge 3$. A circular $r$-uniform hypergraph on $n>r$ vertices is defined as follows: 
Begin with a cycle $C_n$ on $n$ vertices, identified with $\mathbb{Z}_n$.
A proper subset $I\subset \mathbb{Z}_n$ is a \emph{path} in $C_n$ if it induces a connected sub-graph
of the graph $C_n$. The \emph{size} of a path is the number of vertices in the corresponding 
induced  sub-graph.  
We call a hypergraph \emph{circular} if (up to isomorphism) its set of
vertices is equal to $\mathbb{Z}_n$ and its edges are paths of $C_n$ of size $r$. 
We denote by $\Cy_{n,r}$ the circular hypergraph on $n$ vertices whose edge set consists of all paths of size $r$.  Circular hypergraphs have attracted some attention due to the fact that they 
share similar properties with certain classes 
of matrices (see \cite{Quilliot, Tucker}).  Recall that $X_{\Cy_{n,r}}$ denotes the number of vertices 
with non-zero in-degree, after a random orientation on the edges of $\Cy_{n,r}$. \\

We conjecture that circular hypergraphs are such that 
\[ H(X_{\Cy_{n,r}}) \ge H(X_{\Hy}), \text{ for all } \Hy \in \mathcal{D}_{n,n,r} . \]
Moreover, we believe that the following holds true. \\

\begin{conjecture}\label{circular_conjecture}
Fix a positive integer $r\ge 3$. Then for all $n>r$, it holds 
\[ H(X_{\Cy_{n,r}}) \ge \frac{1}{2}\log\left(n/r\right) - O(1). \]
\end{conjecture}


\begin{thebibliography}{99}

\bibitem{Adel_et_al} J.A. Adell, A. Lekuona, Y. Yu. \textit{Sharp Bounds on the Entropy of the Poisson Law and Related Quantities}, IEEE Transactions on Information Theory \textbf{56} (2010) 2299 -- 2306. 

\bibitem{CoverThomas} T. M. Cover, J. A. Thomas. \textit{Elements of information theory}, 2nd Edition, New York, Wiley, 2006.



\bibitem{Johnson} O. Johnson, \textit{Information Theory and the Central Limit Theorem}, Imperial College Press, 2004. 


\bibitem{Massey} J.L. Massey, \textit{On the entropy of integer-valued random variables}, In Proc. 1988 Beijing Int. Workshop on Information Theory, pages C1.1-C1.4, July 1988.

\bibitem{Mow} B.H. Mow. \textit{A tight upper bound on discrete entropy}, IEEE Transactions on
Information Theory \textbf{44} (2) (1998) 775 -- 778. 

\bibitem{Pelekis} C. Pelekis. \textit{Bernoulli trials of fixed parity, random and randomly oriented graphs}, Graphs and Combinatorics \textbf{32} (4) (2016) 1521--1544.

\bibitem{Pel_Schauer} C. Pelekis, M. Schauer. \textit{Network Coloring and Colored Coin Games}. In:  Search Theory: A Game-Theoretic Perspective, Alpern S., et al. (eds), Springer, New York, NY, (2013). 

\bibitem{Quilliot} A. Quilliot, \textit{Circular representation problem on hypergraphs}, 
Discrete Mathematics \textbf{51}, Issue 3, p. 251--264, (1984).

\bibitem{Tucker} A. Tucker, \textit{Matrix characterizations of circular-arc graphs}, Pacific Journal of Mathematics, 
Volume 39, Number 2, p. 535--545, (1971).

\end{thebibliography}
\end{document}